\documentclass{article}
\usepackage{amsmath}
\usepackage{amsfonts}
\usepackage{amsthm}
\usepackage{amssymb, setspace}
\usepackage{mathtools}
\usepackage{lipsum}
\usepackage{manyfoot}
\usepackage{color,graphicx,epsfig,geometry,fancyhdr,hyperref}
\usepackage[T1]{fontenc}
\usepackage{float}
\usepackage{subfigure} 
\usepackage{commath}
\usepackage{bbm}
\usepackage{mathrsfs,fleqn}
\usepackage{pgfplots}
\pgfplotsset{compat=newest}
\newtheorem{theorem}{Theorem}[section]

\newtheorem{lemma}{Lemma}[section]

\newcommand{\N}{\mathbb{N}}

\newcommand{\Z}{\mathbb{Z}}

\newcommand{\R}{\mathbb{R}}

\newcommand{\dnu}{\partial_\nu}

\newcommand{\pdr}{\partial_r}
\newcommand{\dd}{\mathrm{d}}

\newcommand{\grad}{\nabla}

\newcommand{\ov}{\overline}

\newcommand{\paren}[1]{\left( #1 \right) }
\newcommand{\exD}{\mathbb{R}^2 \setminus \overline{D}}
\newcommand{\trB}{B_R \setminus \overline{D}}

\usepackage{comment}
\usepackage{ulem}

\begin{document}

\begin{flushleft}
\Large 
\noindent{\bf \Large Existence of transmission eigenvalues for biharmonic scattering by a clamped planar region}

\end{flushleft}

\vspace{0.2in}

{\bf  \large Isaac Harris\footnote{corresponding author \texttt{harri814@purdue.edu}} }\\
\indent {\small Department of Mathematics, Purdue University, West Lafayette, IN 47907 }\\
\indent {\small Email: \texttt{harri814@purdue.edu} }\\

{\bf  \large Andreas Kleefeld}\\
\indent {\small Forschungszentrum J\"{u}lich GmbH, J\"{u}lich Supercomputing Centre, } \\
\indent {\small Wilhelm-Johnen-Stra{\ss}e, 52425 J\"{u}lich, Germany}\\
\indent {\small University of Applied Sciences Aachen, Faculty of Medical Engineering and } \\
\indent {\small Technomathematics, Heinrich-Mu\ss{}mann-Str. 1, 52428 J\"{u}lich, Germany}\\
\indent {\small Email: \texttt{a.kleefeld@fz-juelich.de}}\\

{\bf  \large Heejin Lee}\\
\indent {\small Zu Chongzhi Center for Mathematics and Computational Sciences, } \\
\indent {\small Duke Kunshan University, Kunshan, Jiangsu Province, China 215316}\\
\indent {\small Email: \texttt{heejin.lee@dukekunshan.edu.cn}}\\

\vspace{0.2in}

\begin{abstract}
\noindent In this paper, we study the so-called clamped transmission eigenvalue problem. This is a new transmission eigenvalue problem that is derived from the scattering of an impenetrable clamped obstacle in a thin elastic plate. The scattering problem is modeled by a biharmonic wave operator given by the Kirchhoff--Love infinite plate problem in the frequency domain. These scattering problems have not been studied to the extent of other models. Unlike other transmission eigenvalue problems, the problem studied here is a system of homogeneous PDEs defined in all of $\mathbb{R}^2$. This provides unique analytical and computational difficulties when studying the clamped transmission eigenvalue problem. We are able to prove that there exist infinitely many real clamped transmission eigenvalues. This is done by studying the equivalent variational formulation. We also investigate the relationship of the clamped transmission eigenvalues to the Dirichlet and Neumann eigenvalues of the negative Laplacian for the bounded scattering obstacle. \\
\end{abstract}



\section{Introduction}
In this paper, we provide an analytical and numerical study of the transmission eigenvalues associated with the scattering of an impenetrable clamped obstacle in a thin elastic plate. The scattering problem is modeled by the Kirchhoff--Love infinite plate equation, which leads to a biharmonic scattering problem. Recently, in \cite{biwellposed,DongLi24} the authors studied the well--posedness for the direct scattering in a thin elastic plate for a clamped obstacle. {\color{black}Here, we study a new transmission eigenvalue problem associated with the aforementioned scattering problem that is different than previously studied problems \cite{interiorTE,ext-ScatTE,TEwObsticale,TE-AnisoCBC}}. This is due to the fact that, this eigenvalue problem is posed on all of $\R^2$ despite the scatterer being a bounded domain. This makes the study of these eigenvalues challenging. For instance, since the problem is posed in an unbounded domain we no longer have the compact embedding of the standard Sobolev spaces. Transmission eigenvalue problems are often non--self--adjoint and nonlinear, which adds to the difficulty in studying them analytically and numerically. The main contribution of this work is proving the existence of the so--called clamped transmission eigenvalues as well as developing new theoretical techniques that can be used for exterior transmission eigenvalue problems \cite{ext-ScatTE,exteriorTE,ext-SteklovEig}. 

As in previous works {\color{black}\cite{interiorTE,ext-ScatTE,TEwObsticale,TE-AnisoCBC}}, we see that this eigenvalue problem is derived from studying the associated inverse scattering problem. While inverse scattering problems have been extensively studied in acoustic and electromagnetic scattering theory, the corresponding problems for biharmonic scattering has only recently begun to attract attention \cite{near-lsmBH,iterative-BH23,DongLi-Unique,LSM-BHclamped,DSM-BH24,teemu-bihar,biharm-stab4potentials}. The aforementioned manuscripts have studied the inverse shape problem given either near or far field data for a few kinds of scatterers. In the case of using far field data, it has been shown that the far field operator associated with a clamped obstacle fails to be injective with a dense range at a transmission eigenvalues; see, e.g., \cite{te-discret}.  Although these inverse problems are relevant in various applications, including non--destructive testing and medical imaging \cite{app1,app2}, both theoretical and numerical investigations remain relatively limited.

Transmission eigenvalue problems have played an important role in inverse scattering theory. The existence of interior transmission eigenvalues for an isotropic acoustic scatterer has been studied in \cite{interiorTE,TEwObsticale,Cakoni-TE}, and the exterior transmission eigenvalues for an isotropic acoustic scatterer has been investigated in \cite{exteriorTE} for a spherically stratified media in three dimensions. Throughout the years, the associated inverse spectral problem has been of interest. This is due to the fact that these eigenvalues can be shown to depend (often monotonically) on the material properties of the scatterer. This has given rise to studying the inverse spectral problem to recover information about the scatterer from the transmission eigenvalues. In general, it has been shown that for acoustic and electromagnetic scatterers that these eigenvalues depend monotonically on the material parameters and geometric properties of the scatterer; see, e.g., \cite{interiorTE,TE-2CBCRafa,TE-AnisoCBC}.

Here, we consider a transmission eigenvalue problem defined through an interior--exterior formulation, where different governing equations hold inside the scatterer and in the unbounded exterior domain, respectively. These two regions are coupled through transmission conditions on the boundary of the scatterer. We study the existence of the real clamped transmission eigenvalues and investigate the relationship between the first clamped transmission eigenvalue and the first Dirichlet eigenvalue of the negative Laplacian for the scattering obstacle. The same transmission eigenvalue problem was also considered in \cite{te-discret}, where it was shown that the transmission eigenvalues can be recovered from biharmonic far field data and that the set of transmission eigenvalues is discrete with some assumption on the wave number.

This paper is organized as follows. In Section \ref{formulation}, we formulate the direct scattering problem and  associated clamped transmission eigenvalue problem. These eigenvalues can be seen as the frequencies where there is an incident wave that produces trivial far field data. In Section \ref{existence}, we prove the existence of real transmission eigenvalues using the techniques introduced in \cite{Cakoni-TE}. Additionally, we study the relationship between the first transmission eigenvalue and the first Dirichlet eigenvalue of the negative Laplacian. We then numerically investigate the interlacing between the clamped transmission eigenvalues along with the Dirichlet and Neumann eigenvalues for the obstacle in Section \ref{numerics}. Lastly, Section \ref{conclusion} provides a summary of our work along with outlooks to other interesting problems in this direction.

\section{Formulation of Transmission Eigenvalue Problem}\label{formulation}
In this section, we introduce the biharmonic scattering problem associated with the clamped transmission eigenvalues. Here, we assume that the clamped obstacle is modeled by a bounded region $D \subset \R^2$ with a {\color{black}Lipschitz} boundary $\partial D$. To illuminate the obstacle, one uses a time--harmonic incident plane wave, which we denote as $u^{\text{inc}}(x) = \text{e}^{\text{i}kx\cdot d}$, where $d$ is the corresponding incident direction such that  $d \in \mathbb{S}^1 = \{ x \in \R^2: |x|=1\}$. For our model, the corresponding scattered field $u^{\text{scat}}$ satisfies the biharmonic scattering problem for a fixed wave number $k>0$,  given by 
\begin{align}
\Delta^2 u^{\text{scat}} - k^4 u^{\text{scat}} = 0 \quad & \text{in } \mathbb{R}^2\setminus \overline{D}, \label{biharmonic} \\
u^{\text{scat}} \big|_{\partial D}=-u^{\text{inc}}  \quad \text{ and } \quad \dnu{u^{\text{scat}}} \big|_{\partial D} &=- \dnu{u^{\text{inc}}}. \label{cbc}
\end{align}
In equation \eqref{cbc}, we let $\nu$ denote the outward unit normal vector on the boundary $\partial D$. It is known that the scattered field $u^{\text{scat}}$ satisfies the radiation conditions \cite{poh-invsource} 
\begin{align}\label{SRC}
\lim_{r \to \infty} \sqrt{r}(\pdr{u^{\text{scat}}} - \text{i} k u^{\text{scat}}) = 0 \quad \text{ and } \quad \lim_{r \to \infty} \sqrt{r}(\pdr{ \Delta u^{\text{scat}}} - \text{i} k \Delta u^{\text{scat}}) = 0  \quad  \quad r=|x|,
\end{align}
which holds uniformly in $\hat{x}=x/|x|$.

As in \cite{DongLi24}, we consider two auxiliary functions $u_H$ and $u_M$ such that
\begin{align*} 
u_H = -\frac{1}{2k^2}\left(\Delta u^{\text{scat}} - k^2 u^{\text{scat}} \right) \quad  \text{ and } \quad u_M = \frac{1}{2k^2}\left(\Delta u^{\text{scat}} + k^2 u^{\text{scat}}\right),
\end{align*}
which implies that $u^{\text{scat}} = u_H + u_M$. Then, we see that $u_H$ satisfies the Helmholtz equation and $u_M$ satisfies the modified Helmholtz equation (i.e. with wave number=$\text{i}k$). This is a direct consequence of the fact that the differential operator $\Delta^2  - k^4 = (\Delta +k^2)(\Delta - k^2)$. 
Therefore, the biharmonic wave scattering problem \eqref{biharmonic}--\eqref{SRC} is equivalent to the following problem:
\begin{align}
\Delta u_H + k^2 u_H = 0 \quad \text{in } \exD  \quad &\text{ and } \quad  \Delta u_M - k^2 u_M = 0 \quad \text{in } \exD \label{vhvmeq1} \\
u_H+u_M = -u^{\text{inc}}   \quad &\text{ and } \quad \dnu(u_H+u_M) = -\dnu u^{\text{inc}} \quad \text{on } \partial D  \label{vhvmeq2}.
\end{align}
where $u_H$ and $u_M$ satisfies the Sommerfeld radiation condition
\begin{align}\label{SRC1}
\lim_{r \to \infty} \sqrt{r}(\pdr{u_H} - \text{i} k u_H) = 0 \quad \text{ and } \quad \lim_{r \to \infty} \sqrt{r}(\pdr{u_M} - \text{i} k u_M) = 0
\end{align}
uniformly in $\hat{x}=x/|x|$. Since the scattered field is assumed to be radiating, it is known that it has the asymptotic behavior
\begin{align*}
u^{\text{scat}}(x, d) = \frac{\text{e}^{\text{i} \pi/4}}{\sqrt{8\pi k}}\cdot \frac{\text{e}^{{\mathrm{i}k}|x|}}{\sqrt{|x|}}u^\infty(\hat{x},d) + \mathcal{O} \left( \frac{1}{|x|^{3/2}} \right), \quad \text{as } |x| \to \infty,
\end{align*}
where $u^\infty(\hat{x}, d)$ is the far field pattern of $u^{\text{scat}}$. Note that, for $k>0$  we assume that both $u_M$ and $\pdr u_M$ decay exponentially as $r  \to \infty$, see for e.g. \cite{DongLi24}. This implies that the far field pattern for the solution to \eqref{biharmonic}--\eqref{SRC} is given by $u^\infty = u_H^\infty$, where $u_H^\infty$ is the far field pattern of $u_H$.

We can now define the so--called far field operator 
$$F: L^2(\mathbb{S}^1) \to L^2(\mathbb{S}^1) \quad \text{ given by } \quad (Fg)(x) = \int_{\mathbb{S}^1} u^\infty (\hat{x}, d) g(d) \, \dd s(d).$$
The far field operator is often used to derive qualitative reconstruction methods like the linear sampling and factorization methods, see manuscripts \cite{Cakoni-Colton-book,kirschbook}. These methods can recover the obstacle $D$ with little a prior information. As discussed in \cite{te-discret}, if we have that the far field operator $F$ is not injective then there is a generic incident field $v$ (i.e. that satisfies the Helmholtz equation in $\R^2$) that produces a trivial far field pattern for \eqref{biharmonic}--\eqref{SRC} or equivalently \eqref{vhvmeq1}--\eqref{SRC1} with $u^{\text{inc}} = v$. From Rellich's Lemma (Theorem 3.5 in \cite{Cakoni-Colton-book}) we have that the corresponding $u_H = 0$ in $\exD$. Therefore, we have that the pair $w=u_M$ in $\exD$ and $v=u^{\text{inc}}$ in $D$ satisfy 
\begin{align}
\Delta v + k^2 v = 0 \quad \text{in } D \quad &\text{and}  \quad   \Delta w - k^2 w = 0 \quad \text{in } \exD,  \label{tep1} \\
v+w = 0 \quad & \text{and } \quad  \dnu(v+w) = 0 \quad \text{on } \partial D. \label{tep2}
\end{align}
It is known that for $k>0$, that both $w$ and $\pdr{w}$ decay exponentially fast as $r \to \infty$. Therefore, we define the clamped transmission eigenvalues to be the wave numbers $k>0$ such that there is a non--trivial solution $(v, w)  \in H^1(D) \times H^1(\exD)$ to the boundary value problem \eqref{tep1}--\eqref{tep2}.

\section{Existence of the Transmission Eigenvalues}\label{existence}
In this section, we will show the existence of real eigenvalues corresponding to the clamped transmission eigenvalue problem \eqref{tep1}--\eqref{tep2} by appealing to Theorem 2.3 in \cite{Cakoni-TE}. This method has been used to prove the existence of other transmission eigenvalues in previous works \cite{interiorTE,TEwObsticale}. It has been shown in \cite{te-discret} that the set of transmission eigenvalues is discrete provided that Re$(k)>0$ and Re$(k^2)>0$. It is also shown in \cite{te-discret} that any transmission eigenvalue such that Re$(k)>0$ must be real--valued. Therefore, we will show the existence of infinitely many positive clamped transmission eigenvalues.

In order to use the results in \cite{Cakoni-TE} to prove the existence of the clamped transmission eigenvalues, it is advantageous to consider \eqref{tep1}--\eqref{tep2} in a truncated domain. With this in mind, we now let $B_R \subset \R^2$ denote the ball of radius $R>0$ centered at the origin such that $D \subset B_R$ where dist$(D,\partial B_R) >0$. With this, we can consider the eigenvalue problem
\begin{align}
\Delta v + k^2 v = 0 \quad \text{in } D \quad &\text{and}  \quad   \Delta w - k^2 w = 0 \quad \text{in } \trB, \label{tep3}\\
v+w = 0 , \quad \dnu(v+w) = 0 \quad \text{on } \partial D \quad &\text{and } \quad   
T_{\text{i}k}w = \dnu w \quad \text{on } \partial B_R, \label{tep4}
\end{align}
where $T_{\text{i}k}$ denotes the exterior Dirichlet--to--Neumann (DtN) for the modified Helmholtz equation in $\R^2 \setminus \overline{B_R}$. The aforementioned DtN mapping $T_{\text{i}k}: H^{1/2}(\partial B_R) \to H^{-1/2}(\partial B_R)$ is defined by 
\begin{align}\label{DtN}
T_{\text{i}k} f = \dnu{u_f} \quad \text{ where } \quad  \Delta u_f - k^2 u_f = 0 \quad \text{in } \,\, \R^2 \setminus \overline{B_R} \quad \text{ with } \,\, u_f \big|_{\partial B_R} = f
\end{align}
where $\nu$ is the unit outward normal to $\partial B_R$. Notice that, $T_{\text{i}k}$ is the standard exterior DtN mapping for the Helmholtz equation with wave number $\text{i}k$. 

We note that, for $k>0$ we have that both $u_f$ and $\pdr u_f$ decay exponentially as $r \to \infty$ so {\color{black}no radiation condition at infinity is required} but we assume that $u_f \in H^1(\R^2 \setminus \overline{B_R})$. First, note that the eigenvalue problems \eqref{tep1}--\eqref{tep2} and \eqref{tep3}--\eqref{tep4} are equivalent. 
\begin{theorem} \label{equiv-tep}
If there is a non--trivial solution to \eqref{tep1}--\eqref{tep2} with eigenvalue $k>0$, then its restriction to $B_R$ is a non--trivial solution to \eqref{tep3}--\eqref{tep4}. Moreover, if there is a non--trivial solution to \eqref{tep3}--\eqref{tep4} with eigenvalue $k>0$, then it can be extended to $\R^2$ by \eqref{DtN} such that the extension satisfies \eqref{tep1}--\eqref{tep2}.
\end{theorem}
\begin{proof}
The proof of the claim is clear by the definition of the DtN mapping \eqref{DtN} and the fact that the associated exterior Dirichlet problem is well--posed for any $f \in H^{1/2}(\partial B_R)$. 
\end{proof} 

This implies that we can prove the existence of eigenpairs $k>0$ and $(v, w) \in  X(D,B_R)$ to \eqref{tep3}--\eqref{tep4} where the Hilbert space 
\begin{align}\label{spacex}
X(D,B_R)= \left\{ (v, w)\in H^1(D) \times H^1(\trB): (v+w)|_{\partial D}=0 \right\}
\end{align}
{\color{black}over the complex plane} with the associated inner--product 
$$\big({(\varphi_1,\varphi_2) \, ; (\psi_1, \psi_2)}\big)_{X(D, B_R)} = (  \varphi_1 , \psi_1 )_{H^1(D)} + (\varphi_2,\psi_2)_{H^1(\trB)}.$$ 
 Therefore, to study the transmission eigenvalue problem in the truncated domain. By appealing to Green's first identity we have that the equivalent variational formulation of the clamped transmission eigenvalue problem \eqref{tep3}--\eqref{tep4} is given by: find eigenpairs $k>0$ and $(v,w) \in X(D, B_R)$ such that
\begin{align}\label{TEeq}
\mathcal{A}_k \big((v,w) \, ; (\psi_1, \psi_2) \big) - k^2 \mathcal{B}\big((v,w) \, ; (\psi_1, \psi_2) \big) = 0 \quad \text{for all} \quad (\psi_1, \psi_2 ) \in X(D, B_R). 
\end{align}
Here, we define the sesquilinear form $\mathcal{A}_k: X(D, B_R)\times X(D, B_R) \longrightarrow \mathbb{C}$ such that
\begin{align}\label{seqformAk}
\mathcal{A}_k\big((v,w) \, ; (\psi_1, \psi_2) \big) = \int_{D}\nabla{v}\cdot \nabla{\overline{\psi_1}} \, \dd x 
+\int_{\trB} {\nabla{w}\cdot \nabla{\ov{\psi_2}} + k^2 w \overline{\psi_2}} \, \dd x  
- \int_{\partial B_R} \overline{\psi_2} T_{\text{i}k}w \, \dd s
\end{align}
and $\mathcal{B}: X(D, B_R)\times X(D, B_R) \longrightarrow \mathbb{C}$ is given by
\begin{align}\label{seqformB}
\mathcal{B}\big((v,w) \, ; (\psi_1, \psi_2) \big)  = \int_D v\overline{\psi_1}\, \dd x.
\end{align}
By the Riesz representation theorem, there exist two bounded linear operators associated with the sesquilinear forms denoted $\mathbb{A}_k$ and $\mathbb{B}: X(D, B_R) \longrightarrow X(D, B_R)$ such that
\begin{align}
\mathcal{A}_k\big((v,w) \, ; (\psi_1, \psi_2) \big) &=  \big(\mathbb{A}_k(v,w) \, ; (\psi_1, \psi_2) \big)_{X(D, B_R)} \label{OpAkDef} 
\end{align}
and 
\begin{align}
\mathcal{B}\big((v,w) \, ; (\psi_1, \psi_2) \big) &=\big(\mathbb{B}(v,w) \, ; (\psi_1, \psi_2) \big)_{X(D, B_R)}.\label{OpBDef}
\end{align}
Notice that $k>0$ is a clamped transmission eigenvalue if and only if the operator $\mathbb{A}_{k}-k^2 \mathbb{B}$ has a non--trivial null space.

As previously stated, in \cite{te-discret} it has been shown that the set of clamped transmission eigenvalues $k>0$ corresponding to \eqref{tep1}--\eqref{tep2} (or equivalently \eqref{tep3}--\eqref{tep4}) is at most discrete. With this, we focus on the existence of clamped transmission eigenvalues in this work. To prove this claim, we will study the analytic properties of the operators $\mathbb{A}_k$ and $\mathbb{B}$ defined in \eqref{OpAkDef}--\eqref{OpBDef}. Motivated by the results in \cite{Cakoni-TE}, we will show that the aforementioned operators satisfy the assumptions of a key result. To this end, we recall Theorem 2.3 of \cite{Cakoni-TE} which will be applied to our eigenvalue problem.
 
\begin{lemma}\label{eig_key_lemma}
Assume that the mapping $k \longmapsto \mathbb{A}_k$ is continuous from $[0,\infty)$ to the set of self--adjoint,  positive definite, bounded linear operators on the Hilbert space $X$ and assume that $\mathbb{B}$ is a self--adjoint, non--negative and compact linear operator on $X$. If we have that there exist two constants $0 \leq \tau_1 < \tau_2  < \infty$ such that
\begin{itemize}
\item[(1)] $\mathbb{A}_{\tau_1}-\tau_1^2 \mathbb{B}$ is positive on $X$, and
\item[(2)] $\mathbb{A}_{\tau_2}-\tau_2^2 \mathbb{B}$ is non--positive on a $m$--dimensional subspace of $X$,
\end{itemize}
then there exists at least $m$--values $k \in [\tau_1,\tau_2]$ such that $\mathbb{A}_{k}-k^2 \mathbb{B}$ has a non--trivial null space. 
\end{lemma}

{\color{black}The operator $\mathbb{A}_0$ will be defined later in equation \eqref{OpA0Def} by studying the limit as $k \to 0^+$ for the DtN mapping.} This result is proven by considering the problem $\lambda_j (k)-k^2=0$ where $\lambda_j(k)$ are the generalized eigenvalues such that $\mathbb{A}_k - \lambda_j(k) \mathbb{B}$ has a non--trivial null space. Since each $\lambda_j(k)$ is continuous with respect to $k \geq 0$, the  positivity and non--positivity requirements in Lemma \ref{eig_key_lemma} implies that 
$$\lambda_j (\tau_1) > \tau_1^2 \quad \text{ and } \quad \lambda_j (\tau_2) \leq \tau_2^2 \quad \text{ for $j = 1, \ldots , m$.}$$
By appealing to the intermediate value theorem each equation $\lambda_j (k)-k^2=0$ has a solution.

With this, we turn our attention to proving that Lemma \ref{eig_key_lemma} can be applied to our eigenvalue problem. This requires us to study the analytical properties of the operators $\mathbb{A}_k$ and $\mathbb{B}$. To this end, notice that from the definition of $T_{\text{i}k}$ we have that 
$$ - \int_{\partial B_R} \overline{\psi_2} T_{\text{i}k}w \, \dd s = \int_{ \R^2 \setminus \overline{B_R} } {\nabla{u_w}\cdot \nabla{\ov{u_{\psi_2}}} + k^2 u_w \overline{u_{\psi_2}}} \, \dd x, $$ 
where $u_{w}$ and  $u_{\psi_2}$ are as defined by \eqref{DtN}. Therefore, we see that the sequilinear form in \eqref{seqformAk} can be written as 
\begin{align*}
\mathcal{A}_k\big((v,w) \, ; (\psi_1, \psi_2) \big) = \int_{D}\nabla{v}\cdot \nabla{\overline{\psi_1}} \, \dd x 
&+\int_{\trB} {\nabla{w}\cdot \nabla{\ov{\psi_2}} + k^2 w \overline{\psi_2}} \, \dd x  \\
&\hspace{0.5in}+ \int_{ \R^2 \setminus \overline{B_R} } {\nabla{u_w}\cdot \nabla{\ov{u_{\psi_2}}} + k^2 u_w \overline{u_{\psi_2}}} \, \dd x.
\end{align*}
We also note that in order to use Lemma \ref{eig_key_lemma} we need to prove that the mapping $k \longmapsto T_{\text{i}k}$ depends continuously on $k>0$. This is given in the following result. 

\begin{theorem}\label{opDtNthm}
The map $k \longmapsto T_{\mathrm{i}k}$ defined by \eqref{DtN} is continuous with respect to $k > 0$. 
\end{theorem}
\begin{proof}
We first note that by the well--posedness of 
$$\Delta u^k_f - k^2 u^k_f = 0 \quad \text{in } \,\, \R^2 \setminus \overline{B_R} \quad \text{ with } \quad u^k_f \big|_{\partial B_R} = f$$
for any $f \in H^{1/2}(\partial B_R)$ there is a constant $C_k > 0$ such that 
$$ \| u^k_f \|_{H^1(\R^2 \setminus \overline{B_R} )} \leq C_k \| f \|_{H^{1/2}(\partial B_R)}.$$
Notice that $C_k$ can depend on $k>0$ but not $f$ and we make the dependance of the solution with respect to $k$ explicit. We will now employ a variational argument to prove that 
$$ \| T_{\mathrm{i}k} - T_{\mathrm{i}\tau} \| \longrightarrow 0 \quad \text{as} \quad \tau \to k$$
where $\| \cdot \|$ is the appropriate operator norm.

Now, to continue we let $u^k_f$ and $u^\tau_f$ satisfy the above boundary value problem with `wave numbers' $k>0$ and $\tau>0$, respectively. Notice that 
$$\Delta (u^k_f - u^\tau_f) =  k^2 u^k_f -\tau^2 u^\tau_f \quad \text{in } \,\, \R^2 \setminus \overline{B_R} \quad \text{ with } \,\, (u^k_f - u^\tau_f) \big|_{\partial B_R} = 0.$$
By Green's first identity, it is clear that for any $\varphi \in H_0^1(\R^2 \setminus \overline{B_R} )$ we obtain the equalities 
\begin{align*}
\int_{\R^2 \setminus \overline{B_R}} \grad (u^k_f - u^\tau_f) \cdot \grad \overline{\varphi} \, \dd x &= - \int_{\R^2 \setminus \overline{B_R}} (k^2 u^k_f -\tau^2 u^\tau_f) \overline{\varphi} \, \dd x \\
&=-\int_{\R^2 \setminus \overline{B_R}} \tau^2 (u^k_f -u^\tau_f)\overline{\varphi}+(k^2-\tau^2)u^k_f \overline{\varphi} \, \dd x.
\end{align*}
By letting $\varphi = (u^k_f -u^\tau_f)$ we get the estimate
\begin{align*}
\| u^k_f - u^\tau_f \|_{H^1(\R^2 \setminus \overline{B_R} )} & \leq \frac{|k^2-\tau^2|}{\min\{1,\tau^2\}} \| u^k_f \|_{L^2(\R^2 \setminus \overline{B_R} )} \leq  \frac{C_k }{\min\{1,\tau^2\}} |k^2-\tau^2| \, \| f \|_{H^{1/2}(\partial B_R)}.
\end{align*}
To proceed, we let $f,g \in H^{1/2}(\partial B_R)$ and consider   
\begin{align*}
-\int_{\partial B_R} \overline{g} \big(T_{\text{i}k} - T_{\text{i}\tau} \big)f \, \dd s &= -\int_{\partial B_R} \overline{u^k_g} \dnu (u^k_f - u^\tau_f)\, \dd s\\
	&=\int_{\R^2 \setminus \overline{B_R}} \grad (u^k_f - u^\tau_f) \cdot \grad \overline{u^k_g}  + (k^2 u^k_f -\tau^2 u^\tau_f) \overline{u^k_g} \, \dd x \\
&=\int_{\R^2 \setminus \overline{B_R}} \grad (u^k_f - u^\tau_f) \cdot \grad \overline{u^k_g}  + (k^2 -\tau^2) u^k_f  \overline{u^k_g} + \tau^2 (u^k_f - u^\tau_f) \overline{u^k_g} \, \dd x.
\end{align*}
Therefore, we {\color{black}obtain the estimate} 
\begin{align*}
\left| \int_{\partial B_R} \overline{g} \big(T_{\text{i}k} - T_{\text{i}\tau} \big)f \, \dd s \right| & \leq  \Bigg\{ \| u^k_f - u^\tau_f \|_{H^1(\R^2 \setminus \overline{B_R} )} + |k^2-\tau^2| \| u^k_f  \|_{L^2(\R^2 \setminus \overline{B_R} )} \\
& \hspace{1.75in}+ \tau^2 \| u^k_f - u^\tau_f \|_{L^2(\R^2 \setminus \overline{B_R} )} \Bigg\} \| u^k_g  \|_{H^1(\R^2 \setminus \overline{B_R} )} .
\end{align*}
By appealing to our above estimates, we have that 
\begin{align*}
\left| \int_{\partial B_R} \overline{g} \big(T_{\text{i}k} - T_{\text{i}\tau} \big)f \, \dd s \right|  \leq \frac{C_k^2}{\min\{1,\tau^2\}} \bigg\{1+\min\{1,\tau^2\}+\tau^2 \bigg\}|k^2-\tau^2|  \| f \|_{H^{1/2}(\partial B_R)} \| g \|_{H^{1/2}(\partial B_R)}. 
\end{align*}
Taking the supremum over $f$ and $g$ with unit norm, we obtain the bound
$$\|T_{\mathrm{i}k} - T_{\mathrm{i}\tau}\| \leq \frac{C_k^2}{\min\{1,\tau^2\}}\bigg\{1+\min\{1,\tau^2\}+\tau^2 \bigg\}|k^2-\tau^2|.$$
Due to the fact that the pre--factor ${\displaystyle \frac{C_k^2}{\min\{1,\tau^2\}}\Big\{1+\min\{1,\tau^2\}+\tau^2 \Big\} }$
is bounded as $\tau \to k$ for any $k>0$ we have that 
$$\|T_{\mathrm{i}k} - T_{\mathrm{i}\tau}\| \leq  \frac{C_k^2}{\min\{1,\tau^2\}}\Big\{1+\min\{1,\tau^2\}+\tau^2 \Big\}|k^2-\tau^2| \longrightarrow 0 \quad \text{as} \quad \tau \to k$$
which proves the claim. 
\end{proof}

With the above analysis we can now prove that the operators $\mathbb{A}_k$ and $\mathbb{B}$ satisfy the assumptions of Lemma \ref{eig_key_lemma}. We prove that each operator satisfies the assumptions in the following two results. 

\begin{theorem}\label{opAkthm}
The map $k \longmapsto \mathbb{A}_k$ defined by \eqref{OpAkDef} is continuous with respect to $k > 0$. Furthermore, the operator $\mathbb{A}_k$ is self--adjoint and coercive on $X(D,B_R)$ for $k > 0$.
\end{theorem}
\begin{proof}
To prove the claim, we start with the continuity with respect to $k > 0$. So assume that $k>0$ and $\tau >0$, then for any $(v, w)$ and $(\psi_1, \psi_2)\in X(D,B_R)$, we have that 
$$\left| \paren{\paren{\mathbb{A}_{{k}} - \mathbb{A}_{{\tau}}}(v,w); (\psi_1, \psi_2)}_{X(D,B_R)} \right|  = \left|   \int_{\trB} \paren{k^2 - \tau^2} w \overline{\psi_2} \,   \dd x - \int_{\partial B_R} \overline{\psi_2} \big(T_{\text{i}k} - T_{\text{i}\tau}\big)w \, \dd s \right|.$$
Therefore, we can obtain the estimate 
$$ \|\mathbb{A}_{{k}} - \mathbb{A}_{{\tau}}\| \leq C \left( |k^2 - \tau^2| + \|T_{\text{i}k} - T_{\text{i}\tau} \| \right),$$
where $\| \cdot \|$ denotes the appropriate operator norms. By Theorem \ref{opDtNthm} we have that $ \|\mathbb{A}_{{k}} - \mathbb{A}_{{\tau}}\| \longrightarrow 0$ as $\tau \to k$, proving the continuity. 

Now, to prove that the operator is self--adjoint, by appealing to the definition of the sesquilinear form in \eqref{seqformAk} we notice that 
\begin{align*}
\paren{\mathbb{A}_{k}(v,w);(v,w)}_{X(D,B_R)} &= \int_D |\nabla v|^2  \, \dd x + \int_{\trB} {|\nabla w|^2 +k^2  {\color{black}|w|^2}} \, \dd x -  \int_{\partial B_R} \overline{w} T_{\text{i}k} w \, \dd s  \\
&=\int_D |\nabla v|^2  \, \dd x + \int_{\trB} {|\nabla w|^2 +k^2  {\color{black}|w|^2}} \, \dd x + \int_{\exD} {|\nabla u_w|^2 + k^2 |u_w|^2}\dd x.
\end{align*}
Again, we note that the extension $u_w$ is defined by \eqref{DtN}. Since the above expression is non--negative for all $(v, w) \in X(D,B_R)$ with $k>0$ and $X(D,B_R)$ is assumed to be a complex Hilbert space, we have that $\mathbb{A}_{k}$ is self--adjoint.
With the above equality, we can prove the coercivity of the operator $\mathbb{A}_{k}$ for all $k>0$. For a contradiction, assume $\mathbb{A}_{k}$ is not coercive for some $k>0$, then there exists a sequence $\left\{ (v_n, w_n) \right\}_{n \in \mathbb{N}}$ in the Hilbert space $X(D,B_R)$ with $\|(v_n, w_n)\|_{X(D,B_R)} = 1$ for each $n$ such that 
\begin{align*}
\frac{1}{n} \geq    \paren{\mathbb{A}_{k}(v_n ,w_n ) \, ;(v_n ,w_n)}_{X(D,B_R)}  \geq \|\nabla v_n \|^2_{L^2(D)} + \|\nabla w_n\|^2_{L^2(\trB)} + k^2 \|w_n\|^2_{L^2(\trB)}.
\end{align*}
Since the sequence is bounded, we have that (up to a subsequence) it is weakly convergent to some $(v , w ) \in {X(D,B_R)} $. This implies that $w_n \to w=0$ in $H^1(\trB)$ and that $| \grad v_n| \to |\grad v| = 0$ in $L^2(D)$ as $n \to \infty$. Therefore, we see that the limiting function $v$ is constant. By appealing to the compact embedding $H^1(D)$ into $L^2(D)$ we have that (up to a subsequence) $v_n$ is strongly convergent in $H^1(D)$.  The boundary condition $v=-w$ on $\partial D$ implies that  $v=0$ in $D$. The convergence $(v_n, w_n) \to (0,0)$ in $X(D,B_R)$ as $n \to \infty$ contradicts the fact that $\|(v_n,w_n)\|_{X(D,B_R)}=1$. This proves the coercivity and therefore proves the claim.
\end{proof}

\begin{theorem}\label{opBthm}
The operator $\mathbb{B}$ defined by \eqref{OpBDef} is a self--adjoint, compact and non--negative. 
\end{theorem}
\begin{proof}
It is clear that from the sequilinear form \eqref{seqformB} that $\mathbb{B}$ given by \eqref{OpBDef} is self--adjoint and non--negative since $\mathcal{B}\big((v,w) \, ; (v,w) \big) = \|v\|^2_{L^2(D)}$. For the compactness, by definition we can easily obtain that $\|\mathbb{B}(v, w)\|_{X(D,B_R)}\leq \|v\|_{L^2(D)}$. The compact embedding of $H^1(D)$ into $L^2(D)$ proves the claim, since $D$ is a bounded domain with a {\color{black} Lipschitz} boundary. 
\end{proof}

Lastly, we need to consider the operator $\mathbb{A}_0$. To this end, we must study the DtN mapping further, since as it is defined we require $k>0$. In particular, we need to resolve the limit as $k \to 0^+$. With this in mind, we note that  it has been shown in \cite{FEM-wDtN1,BEM-FEMtransproblem} that 
$$T_{\text{i}k}f = \sum\limits_{n=-\infty}^{\infty} \text{i}k \frac{ H^{(1)'}_{|n|}(\text{i}k R)}{H^{(1)}_{|n|}(\text{i}k R)} f_n \text{e}^{\text{i}n \theta} \quad \text{where} \quad f(\theta) =\sum\limits_{n=-\infty}^{\infty} f_n \text{e}^{\text{i}n \theta}$$ 
where, $H^{(1)}_{\ell}$ denotes the first kind Hankel function of order $\ell$ and $f_n$ are the Fourier coefficients for $f \in H^{1/2}(\partial B_R)$. Note that above we have used the fact that 
$$H^{(1)}_{-\ell}(z) = (-1)^\ell H^{(1)}_{\ell}(z) \quad \text{ for all } \quad \ell \in \N .$$
Recall that $H^{p}(\partial B_R)$ can be identified with $H^{p}(0,2\pi)$ via the Fourier series expansion with the associated norm for any $p \in \R$. As in \cite{resonances-wDtN} the well--known recurrence relationship of the Hankel functions and their derivatives (see for e.g. \cite{bessel-webpage}) gives that
$$\text{i}k \frac{ H^{(1)'}_{|n|}(\text{i}k R)}{H^{(1)}_{|n|}(\text{i}k R)} = \gamma_{|n|}(k) -\frac{|n|}{R}, \quad \text{ where we let} \quad \gamma_{|n|}(k) =  \text{i}k  \frac{H^{(1)}_{|n|-1}(\text{i}k R)}{H^{(1)}_{|n|}(\text{i}k R)}.$$
With this, we define 
\begin{align}\label{DtNk=0}
T_{0}: H^{1/2}(\partial B_R) \to H^{-1/2}(\partial B_R) \quad\text{such that } \quad T_{0} f = - \sum\limits_{n=-\infty}^{\infty}  \frac{|n|}{R} f_n \text{e}^{\text{i}n \theta}.
\end{align}

The operator defined in \eqref{DtNk=0} can be seen as the DtN mapping for $k=0$. In order to prove this claim, we show that $\gamma_{|n|}(k) \to 0$ as $k \to 0^+$ for all $n \in \Z$. Indeed, by the asymptotic limits of the Hankel functions (see for e.g. \cite{bessel-webpage}) we have that 
$$H^{(1)}_{0}(z) \sim \frac{2 \text{i}}{\pi} \ln z \quad \text{ and } \quad H^{(1)}_{\ell}(z)  \sim - \frac{\text{i}}{\pi} \ell! \left( \frac{z}{2}\right)^{- \ell}$$ 
for all $\ell \in \N$ as $z \to 0$. Therefore, after some simple calculations we have that 
$$\gamma_{0}(k)  \sim \frac{1}{R \ln (\text{i}k R)}\, , \quad \gamma_{|\pm 1|}(k)  \sim  k^2 R \ln (\text{i}k R) \quad \text{ and } \quad \gamma_{|n|}(k)  \sim - \frac{k^2 R}{2 |n|} \quad \text{ when $|n| >1$ }$$ 
as $k \to 0^+$. With the above asymptotics we have the following result.

\begin{theorem}\label{opDtNthm2}
The map $k \longmapsto T_{\mathrm{i}k}$ defined by \eqref{DtN} satisfies 
$$ \|T_{\mathrm{i}k} - T_{0} \| \longrightarrow 0  \quad \text{ as } \quad k \to 0^+$$
where $\| \cdot \|$ is the operator norm from $H^{1/2}(\partial B_R) \to H^{-1/2}(\partial B_R)$. 
\end{theorem}
\begin{proof}
Here, we can prove the claim by appealing to the series representation of the DtN mappings. Indeed, from the above discussion we have that 
$$(T_{\mathrm{i}k} - T_{0})f =  \sum\limits_{n=-\infty}^{\infty} \gamma_{|n|}(k) f_n \text{e}^{\text{i}n \theta} \quad \text{ where again we let} \quad \gamma_{|n|}(k) =  \text{i}k  \frac{H^{(1)}_{|n|-1}(\text{i}k R)}{H^{(1)}_{|n|}(\text{i}k R)}.$$
Now, for the space $H^{-1/2}(\partial B_R)$, it can be identified with $H^{-1/2}(0,2\pi)$, which implies that
\begin{align*}
2\pi \big\| (T_{\mathrm{i}k} - T_{0})f \big\|^2_{H^{-1/2}(0,2\pi)} &=  \sum\limits_{n=-\infty}^{\infty} \frac{|\gamma_{|n|}(k)|^2}{(1+|n|^2)^{1/2}} |f_n|^2 \\
	&=\sum\limits_{n=-N}^{N} \frac{|\gamma_{|n|}(k)|^2}{(1+|n|^2)^{1/2}} |f_n|^2 + \sum\limits_{|n|\geq N+1} \frac{|\gamma_{|n|}(k)|^2}{(1+|n|^2)^{1/2}} |f_n|^2
\end{align*}
for any $N \in \N$. Let $N$ be fixed and let $k$ be sufficiently small, then we have that
\begin{align*}
2\pi \big\| (T_{\mathrm{i}k} - T_{0})f \big\|^2_{H^{-1/2}(0,2\pi)} &\leq \max_{0\leq |n| \leq N} \left\{ |\gamma_{|n|}(k)|^2 \right\}\sum\limits_{n=-N}^{N} \frac{1}{(1+|n|^2)^{1/2}}  |f_n|^2 \\
	& \hspace{1.5in}+ C k^4 \sum\limits_{|n|\geq N+1} \frac{1}{|n|^2(1+|n|^2)^{1/2}} |f_n|^2.
\end{align*}
Note, that we have used the asymptotic relationship 
$$\gamma_{|n|}(k)  \sim - \frac{k^2 R}{2 |n|} \quad \text{ when $|n| >1$ } \quad \text{as} \quad k\to 0^+.$$
From the above estimate, we have that 
$$2\pi \big\| (T_{\mathrm{i}k} - T_{0})f \big\|^2_{H^{-1/2}(0,2\pi)} \leq C \left( \max_{0\leq |n| \leq N} \left\{ |\gamma_{|n|}(k)|^2 \right\} +  k^4 \right) \|f\|^2_{H^{1/2}(0,2\pi)}.$$
Therefore, we can conclude that 
$$ \big\| (T_{\mathrm{i}k} - T_{0}) \big\|^2 \leq C \left( \max_{0\leq |n| \leq N} \left\{ |\gamma_{|n|}(k)|^2 \right\} + k^4 \right) \quad \text{as} \quad k\to 0^+.$$
This proves the claim since both $|\gamma_{|n|}(k)|^2$ and $k^4$ tend to zero as $k\to 0^+$. 
\end{proof}

With the above result in Theorem \ref{opDtNthm2} we are now able to begin showing the existence of the clamped transmission eigenvalues.  Notice, that Theorem \ref{opDtNthm2} implies that 
$$ \mathbb{A}_k  \longrightarrow   \mathbb{A}_0   \quad \text{ as} \quad  k \to 0^+$$
in the operator norm where by the Riesz representation theorem the operator $\mathbb{A}_0 : X(D,B_R) \to X(D,B_R)$ is given by the variational equality  
\begin{align} \label{OpA0Def}
\big( \mathbb{A}_0 (v,w) \, ; (\psi_1, \psi_2) \big)_{X(D,B_R)}= \int_{D} \nabla{v}\cdot \nabla{\overline{\psi_1}} \, \dd x 
+\int_{\trB} {\nabla{w}\cdot \nabla{\ov{\psi_2}}}  \, \dd x  - \int_{\partial B_R} \overline{\psi_2} T_{0}w \, \dd s
\end{align}
i.e. the mapping $k \longmapsto \mathbb{A}_k$ is contiuous for all $k \geq 0$. Notice that the above analysis along with Theorems \ref{opAkthm} and \ref{opBthm} imply that the operators $\mathbb{A}_k$ and $\mathbb{B}$ satisfy the assumptions of Lemma \ref{eig_key_lemma}. To prove the existence of the clamped transmission eigenvalues, we now turn our attention to showing that there exists $\tau_1 \in [0, \infty)$ such that $\mathbb{A}_{\tau_1}-\tau_1^2\mathbb{B}$ is a positive operator and that there exists $\tau_2  \in [0, \infty)$ such that $\mathbb{A}_{\tau_2}-\tau_2^2\mathbb{B}$ is non--positive on some subset of $X(D,B_R)$. To this end, notice that 
\begin{align*}
\paren{\mathbb{A}_0(v,w)\, ; (v,w)}_{X(D,B_R)}
&= \int_D|\nabla v|^2 \dd x + \int_{\trB} |\nabla w|^2 - \int_{\partial B_R} \overline{w} T_0 w \dd s\\
& \geq \|\nabla v\|^2_{L^2(D)} + \|\nabla w\|^2_{L^2(\trB)} + C\|w\|^2_{H^{1/2}(\partial B_R)}\\
& \geq \min\{1, C\} \paren{\|\nabla v\|^2_{L^2(D)} + \|\nabla w\|^2_{L^2(\trB)} + \|w\|^2_{H^{1/2}(\partial B_R)}},
\end{align*}
where we have used  
$$  - \int_{\partial B_R} \overline{w} T_0 w \dd s \geq C \|w\|^2_{H^{1/2}(\partial B_R)}$$
for some constant $C>0$ independent of $w$ (see Section 5.3 of \cite{Cakoni-Colton-book}). We now wish to prove that $\mathbb{A}_0$ is coercive which would imply that $\mathbb{A}_{0} - {0}^2 \mathbb{B}$ is positive on $X(D,B_R)$.

\begin{theorem}\label{opA0thm}
The operator $\mathbb{A}_0$ defined by \eqref{OpA0Def} is coercive on $X(D,B_R)$.
\end{theorem}
\begin{proof}
We proceed by way of contradiction. Therefore, assume that there does not exist a constant $C>0$ such that 
$$\paren{\mathbb{A}_0(v,w) \, ; (v,w)}_{X(D,B_R)} \geq C \|(v,w)\|_{X(D,B_R)}^2,$$ 
then there exists a sequence $\left\{ (v_n, w_n) \right\}_{n \in \mathbb{N}} \in X(D,B_R)$ with $\|(v_n, w_n)\|_{X(D,B_R)} = 1$ for every $n \in \N$ such that 
\begin{align*}
\frac{1}{n} \geq \paren{\mathbb{A}_0(v_n,w_n) \, ; (v_n,w_n)}_{X(D,B_R)} \geq \alpha \paren{\|\nabla v_n \|^2_{L^2(D)} + \|\nabla w_n \|^2_{L^2(\trB)} + \|w_n\|^2_{H^{1/2}(\partial B_R)}} 
\end{align*}
for some $\alpha >0$. Since the sequence is bounded, there exists a subsequence that converges weakly to $(v,w)$ in $X(D,B_R)$. It follows that 
\begin{align*}
\nabla v = 0 \; \text{ in } D \; \text{ and } \nabla w = 0 \; \text{ in } \trB \text{ with } w=0 \; \text{ on } \partial B_R.
\end{align*}
Then, we have that $w_n$ converges strongly to $w=0$ in $H^1(\trB)$ by the compact embedding of $H^1(\trB)$ into $L^2(\trB)$. In addition, from the boundary condition $v|_{\partial D} = 0$ we obtain that $v$ also vanishes in $D$. The compact embedding of $H^1(D)$ into $L^2(D)$, implies that $v_n$ converges strongly to $0$, which contradicts to  $\|(v_n, w_n)\|_{X(D,B_R)} = 1$. Therefore, the operator $\mathbb{A}_0$ is coercive. 
\end{proof}

This was the last piece needed to prove the main result of the manuscript, i.e., that there exist infinitely many real clamped transmission eigenvalues.

\begin{theorem}\label{TE-exist}
There exist infinitely many positive clamped transmission eigenvalues corresponding to \eqref{tep3}--\eqref{tep4} or equivalently \eqref{tep1}--\eqref{tep2}. 
\end{theorem}
\begin{proof}
To prove the claim, notice that we have shown that the operators $\mathbb{A}_{k}$ and $\mathbb{B}$ satisfy the assumption of Lemma \ref{eig_key_lemma}. We have already shown that $\tau_1=0$ satisfies that $\mathbb{A}_{\tau_1} - \tau_1^2 \mathbb{B}$ is coercive (i.e. positive) on $X(D,B_R)$. The last piece we need is to prove that there is a $\tau_2>0$ such that $\mathbb{A}_{\tau_2} - \tau_2^2 \mathbb{B}$ is non--positive on some $m$--dimensional subspace. This would imply the existence of at least $m$ clamped transmission eigenvalues. 

We now construct the relevant subspace by letting $B_\varepsilon^j$ for $j=1 , \ldots , M(\varepsilon)$ denote disjoint balls of radius $\varepsilon>0$. Here, $M(\varepsilon) \in \N$ is the largest number of disjoint disks $B_\varepsilon^j$ such that ${B_\varepsilon^j} \subset D$ for each $j$. Now let 
$$\tau_2 = \sqrt{\lambda_\varepsilon}, \quad \text{ where $\lambda_\varepsilon$ is the first Dirichlet eigenvalue for the disk $B_\varepsilon^j$.}$$
Note that since all the disks $B_\varepsilon^j$ have the same radius (with different centers) they all have the same set of Dirichlet eigenvalues. Therefore, there exists a corresponding non--trivial Dirichlet eigenfunctions 
$$v_{\varepsilon}^j \in H^1_0(B_\varepsilon^j),\quad \text{ such that } \quad \Delta v_{\varepsilon}^j +\tau_2^2 v_{\varepsilon}^j = 0 \quad \text{in $B_{\varepsilon}^j$}$$ 
for each $j=1 , \ldots , M(\varepsilon)$.

To construct an $M(\varepsilon)$--dimensional subspace of $X(D,B_R)$, we define $v_j \in H_0^1(D)$ such that 
$$v_j = v_{\varepsilon}^j \,\, \, \text{in $B_{\varepsilon}^j$} \quad \text{and} \quad v_j=0 \,\, \,  \text{in $D \setminus \ov{B_{\varepsilon}^j}$.}$$ 
Since $B_\varepsilon^j$ are disjoint, we have that $v_j$ for $j=1, \ldots, {M(\varepsilon)}$ have disjoint support. This implies that the functions $v_j$ are orthogonal and therefore linearly independent in $H^1(D)$. Let us now define the subspace 
$$W_{M(\varepsilon)} = \text{Span} \Big\{ (v_j , 0) \, : \,  j=1, \ldots, {M(\varepsilon)} \Big\} \subset X(D,B_R).$$ 
It is clear that for $(v_j,0) \in W_{M(\varepsilon)}$ we have that 
$$\mathcal{A}_{\tau_2} \big((v_j,0)\, ; (v_j,0) \big) - \tau_2^2 \mathcal{B}\big((v_j,0) \, ; (v_j,0) \big) = \int_D {|\nabla v_j|^2 - \tau_2^2 |v_j|^2} \, \dd x = 0.$$
Thus, we have that for any $(v,0) \in W_{M(\varepsilon)}$ that 
$$\mathcal{A}_{\tau_2} \big((v,0)\, ; (v,0) \big) - \tau_2^2 \mathcal{B}\big((v,0) \, ; (v,0) \big) = \int_D {|\nabla v|^2 - \tau_2^2 |v|^2} \, \dd x =   0,$$
i.e., $\mathbb{A}_{\tau_2} - \tau_2^2 \mathbb{B}$ is non--positive on $W_{M(\varepsilon)}$. By Lemma \ref{eig_key_lemma}, this implies that there are at least $M(\varepsilon)$ clamped transmission eigenvalues. Using the fact that $M(\varepsilon) \to \infty$ as $\varepsilon \to 0$, we conclude that there are infinitely many positive clamped transmission eigenvalues. 
\end{proof}

With Theorem \ref{TE-exist}, we have shown the existence of infinitely many real clamped transmission eigenvalues. One open question is to determine how the clamped transmission eigenvalues depend on the scatterer $D$. It is conjectured in \cite{te-discret} that these clamped transmission eigenvalues are monotonically decreasing with respect to the measure of $D$. This conjecture is backed up by some numerical evidence. We know this fact to be true for the Dirichlet eigenvalues of the region $D$. Now, we aim to investigate the relationship between the first real clamped transmission eigenvalue and the first Dirichlet eigenvalue of the negative Laplacian in $D$.

\begin{theorem}\label{upperbound}
Let $k_1$ be the first eigenvalue corresponding to \eqref{tep3}--\eqref{tep4} or equivalently \eqref{tep1}--\eqref{tep2} and $\lambda_1$ be the first Dirichlet eigenvalue of the negative Laplacian in $D$.
Then, we have that $k_1^2 \leq \lambda_1.$
\end{theorem}
\begin{proof}
To begin, we can work similarly to the proof of the previous result. We consider the first Dirichlet eigenvalue of the negative Laplacian in $D$ which we denote $\lambda_1$ with corresponding normalized eigenfunction $v \in H^1_0(D)$.
Now, we define the positive value $\tau_2 = \sqrt{\lambda_1}$. 
Therefore, we have that 
$$W = \text{Span} \big\{ (v , 0)  \big\} \subset X(D,B_R)$$ 
is a one dimensional subspace of $X(D,B_R)$. 
By Green's first identity, we obtain that 
$$\mathcal{A}_{\tau_2} \big((v,0)\, ; (v,0) \big) - \tau_2^2 \mathcal{B}\big((v,0) \, ; (v,0) \big) = \int_D |\nabla v|^2 - \tau_2^2 |v|^2 \, \dd x =   0$$
i.e. $\mathbb{A}_{\tau_2} - \tau_2^2 \mathbb{B}$ is non--positive on $W$. Again by appealing to Lemma \ref{eig_key_lemma}, we have that there is at least one clamped transmission eigenvalue in the interval $( 0 , \tau_2]$. In particular, this would imply that the smallest (i.e. first) clamped transmission eigenvalue $k_1$ satisfies $k_1^2 \leq \lambda_1$, proving the claim. 
\end{proof}

This result provides a connection between the first clamped transmission eigenvalues and the first Dirichlet eigenvalue of $D$. This does not prove that the first clamped transmission eigenvalue is monotonically decreasing with respect to the measure of $D$, but does give some analytical evidence to this conjecture. This is due to the fact that the first Dirichlet eigenvalue is monotonically decreasing as a function of the measure of $D$. 

To conclude this section, we want to give a numerical example of Theorem \ref{upperbound}. Therefore, we will let $D$ be given by an ellipse such that 
$$\partial D = \big(\cos t \, , \,  \varepsilon \sin t\big)^\top \,\,\text{ for $t \in [0,2\pi)$} \quad \text{for $0<\varepsilon \leq 1$}.$$ 
Assuming we have a `thin' ellipse (i.e. $\varepsilon \ll 1$) it has been proven in \cite{Deig-asymp} (see equation (1.2)) that the first Dirichlet eigenvalue of the negative Laplacian has the asymptotic expansion 
$$\lambda_1 (\varepsilon)= \frac{\pi^2}{4\varepsilon^2} + \frac{\pi}{2\varepsilon}+\frac{3}{4}+ \left( \frac{11}{8\pi}+\frac{\pi}{12}\right)\varepsilon +\mathcal{O}(\varepsilon^2) \quad \text{as} \quad  \varepsilon \to 0.$$
Using the above formula for an approximation of the first Dirichlet eigenvalue along with the reported clamped transmission eigenvalues in \cite{te-discret} we can provide some numerical validation for Theorem \ref{upperbound}. For the case when $\varepsilon=1$, we use the well known explicit value for $\lambda_1 (1) \approx 5.78323$. 
\begin{table}[!ht]
\centering
 \begin{tabular}{r|c|c}
 $\varepsilon$ & $k_1$ & $\sqrt{\lambda_1 (\varepsilon)}$ \\
\hline
$0.5$ &  2.40418 & 3.75645\\
$0.6$ &  2.14377 & 3.26214\\
$0.7$ &  1.95646 & 2.91875\\
$0.8$ &  1.81492 & 2.66990\\
$1.0$ &  1.61464 & 2.40483\\
\hline
\end{tabular}
\caption{\label{ellipse} Comparison of the first clamped transmission eigenvalue and first Dirichlet eigenvalue for an elliptical scatterer. }
\end{table}
From Table \ref{ellipse}, we see that Theorem \ref{upperbound} has been validated. Also, as previously stated, we see that the clamped transmission eigenvalue seems to be monotonically decreasing as a function of the measure of $D$ since in our setup meas$(D)=\pi \varepsilon$. 

\section{Numerical Investigation}\label{numerics}
In this section, we will provide some numerical calculations for our clamped transmission eigenvalue problem. We also provide numerical evidence for an `interlacing conjecture'. Notice, that by Theorem \ref{upperbound} we have that $0<k_1^2 \leq \lambda_1$ where $\lambda_1$ is the first Dirichlet eigenvalue of the negative Laplacian in $D$. This can also be seen as 
$$\mu_1 = 0<k_1^2 \leq \lambda_1 \quad \text{where $\lambda_1$=1st Dirichlet eigenvalue and $\mu_1$=1st Neumann eigenvalue.}$$ 
This brings up the natural question, are the clamped transmission eigenvalues interlaced with the Dirichlet and Neumann eigenvalues? It is known that the Dirichlet and Neumann eigenvalues are interlacing, i.e. $\mu_j \leq \lambda_j$ where $\lambda_j$ is the $j$th Dirichlet eigenvalue and $\mu_j$ is the $j$th Neumann eigenvalue. Here we will investigate the interlacing conjecture: for any $j \in \N$
$$\mu_j \leq k_j^2 \leq \lambda_j \quad \text{where $\lambda_j$=$j$th Dirichlet eigenvalue and $\mu_j$=$j$th Neumann eigenvalue}$$
via our numerical calculations. 

To this end, we will use the method described in Section 5 of \cite{te-discret} to numerically compute the clamped transmission eigenvalues for a smooth scatterer $D$. This uses boundary integral equations to compute the clamped transmission eigenvalues. Similarly,  we will use boundary integral equations to compute the corresponding Dirichlet and Neumann eigenvalues. Therefore, let $u_{\text{Dir}}$ and $u_{\text{Neu}}$ be the corresponding Dirichlet and Neumann eigenfunctions for the region $D$. 

In order to compute the Dirichlet and Neumann eigenvalues, we will appeal to the method introduced in \cite{beyn}. Therefore, we will write the eigenvalue problems as a nonlinear operator values problem just as in \cite{klee2020}. With this in mind, we make the ansatz that 
$$u_{\text{Dir}}(x)=\mathrm{DL}_{\sqrt{\lambda}}\varphi(x) \quad \text{and} \quad u_{\text{Neu}}(x)=\mathrm{SL}_{\sqrt{\mu}}\psi(x)$$
for some densities $\varphi$ and $\psi$. Here, we let the double and single--layer operator be defined as
$$\mathrm{DL}_{\tau}\varphi(x)=\int_{\partial D} \partial_{\nu(y)} \Phi_\tau (x,y)\varphi(y)\,\mathrm{d}s(y) \quad \text{for $x \in D$}$$
and 
$$\mathrm{SL}_{\tau}\psi(x)=\int_{\partial D} \Phi_\tau (x,y) \psi(y)\,\mathrm{d}s(y) \quad \text{for $x \in D$}$$
where 
$${\displaystyle \Phi_\tau(x,y) = \frac{\text{i}}{4} H^{(1)}_0(\tau |x-y| ) } \quad \text{ for all } \quad x \neq y,$$
is the radiating fundamental solution for the Helmholtz equation in $\R^2$.  {\color{black}Here,} $\lambda$ and $\mu$ are the corresponding Dirichlet and Neumann eigenvalues when
$$u_{\text{Dir}}(x)=0 \quad \text{and } \quad \partial_{\nu(x)} u_{\text{Neu}} (x) = 0 \quad \text{for $x \in \partial D$},$$
i.e., by the well--known jump relations we have that there are non--trivial densities satisfying 
\begin{align}
\left(-\frac{1}{2}I+\mathrm{D}_{\sqrt{\lambda}}\right)\varphi(x) = 0 \quad \text{and} \quad  \left(\frac{1}{2}I+\mathrm{D}^\top_{\sqrt{\mu}}\right)\psi(x) = 0 \quad \text{for $x \in \partial D$.} \label{BIEs}
\end{align}
Here, we have that the above boundary integral operator are defined as 
$$\mathrm{D}_{\tau}\varphi(x)=\int_{\partial D} \partial_{\nu(y)} \Phi_\tau (x,y)\varphi(y)\,\mathrm{d}s(y) \quad \text{for $x \in \partial D$}$$
and 
$$\mathrm{D}^\top_{\tau}\psi(x)=\int_{\partial D} \partial_{\nu(x)} \Phi_\tau (x,y) \psi(y)\,\mathrm{d}s(y) \quad \text{for $x \in \partial D$}$$
with $I$ denoting the identity operator in the appropriate Sobolev Space for the unknown densities. 

The above boundary integral equations in \eqref{BIEs} are discretized in the same way as the one for the computation of clamped transmission eigenvalues via boundary element collocation method. Precisely, we use 40 faces (120 collocation nodes) to approximate the integral operators. With this we are able to compute the first five eigenvalues for each system to check the interlacing conjecture by using the  algorithm in \cite{beyn}. First, we consider the case when $D$ is given by the unit disk. In Table \ref{UnitDisk} we provide the first five clamped transmission eigenvalues $k_j$  (TE) along with the square root of the Dirichlet eigenvalues $\sqrt{\lambda_j}$ (DE) and square root of the Neumann eigenvalues $\sqrt{\mu_j}$ (NE). 
\begin{table}[H]
\centering
 \begin{tabular}{r|r|r|r|r|r|r}
\hline
\hline
                     & NE  & 0.00000 & 1.84119 & 1.84119 & 3.05424 & 3.05424\\
Unit Disk	    & TE  & 1.61464 & 3.05164 & 3.05164 & 4.36453 & 4.36453\\
		    & DE & 2.40483 & 3.83171 & 3.83171 & 5.13563 & 5.13562\\
\hline
 \end{tabular}
 \caption{The first five Dirichlet, Neumann and Clamped Transmission Eigenvalues for a unit disk.} \label{UnitDisk}
\end{table}

Notice that the computed Dirichlet and Neumann eigenvalues in Table \ref{UnitDisk} match up with the known eigenvalues computed via separation of variables. This implies that the aforementioned method is an accurate way to compute the Dirichlet and Neumann eigenvalues via boundary integral equations. We also notice that the interlacing conjecture seems to hold for the unit disk from our calculations.

We will now check our conjecture for different scatterers. Therefore, we will apply the numerical procedure described above for more complex shaped scatterers. This requires computing the boundary integral equations for the scatterers defined by their parametric representations. In Table \ref{scatterers}, we provide the representations for the Peanut--shaped, Star--shaped and  Kite--shaped scatterers. Also, the visual representation of the scatterers is given in Figure \ref{fig:scatterers}. 
  
\begin{table}[h]
	\centering
	\begin{tabular}{ l l }
\hline
\hline
		Scatterer     &  Parameterization \\
\hline 
\hline
		\vspace{-2.0ex}\\
		Star--shaped  & $x(t)={\displaystyle 0.25(3\cos(5t)/10+2) \big(\cos(t),\sin(t) \big)^\top}$
		\vspace{1.5ex} \\ 
		Peanut--shaped   & $x(t)={\displaystyle 0.5\sqrt{3\cos(t)^2+1} \big(\cos(t),\sin(t) \big)^\top}$  
		\vspace{1.5ex}\\ 
		Kite--shaped  & $x(t)={\displaystyle \big(0.75\cos(t)+0.3 \cos(2t),\sin(t) \big)^\top}$
		\vspace{1.5ex}\\
\hline
	\end{tabular}\caption{The boundary parameterizations of $\partial D=x(t)$ for $t\in [0,2\pi)$.} \label{scatterers}
\end{table}

\begin{figure}[h]
	\centering	
	\subfigure[Star--shaped scatterer]{\includegraphics[width=0.32\textwidth]{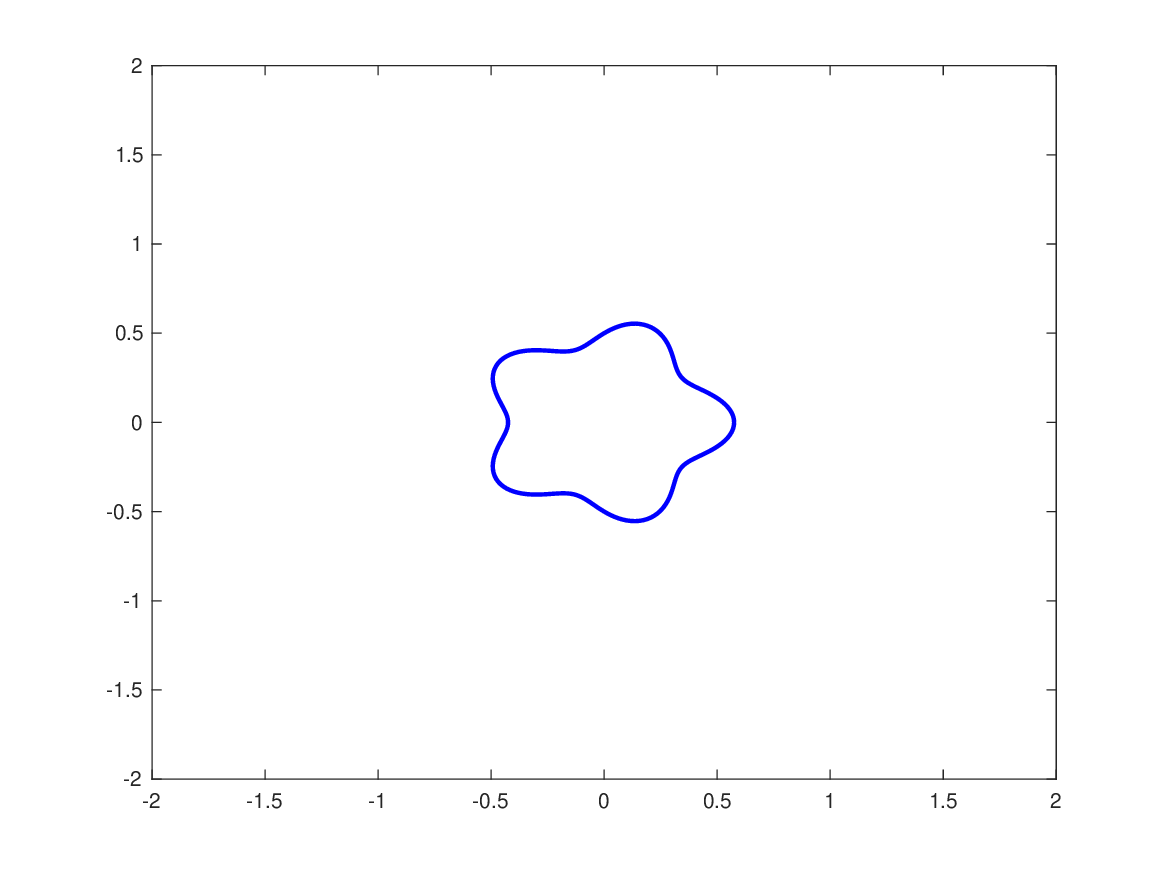}} 
	\subfigure[Peanut--shaped scatterer]{\includegraphics[width=0.32\textwidth]{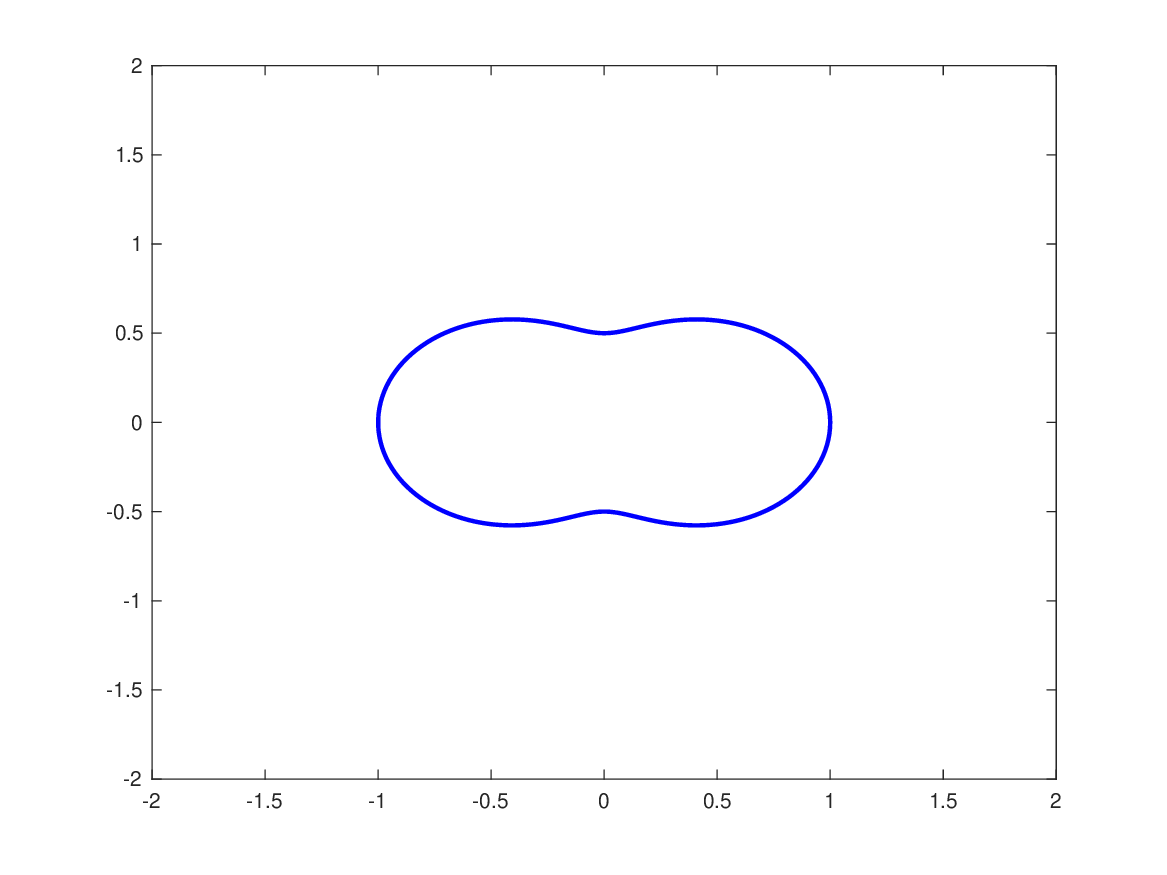}}
	\subfigure[Kite--shaped scatterer]{\includegraphics[width=0.32\textwidth]{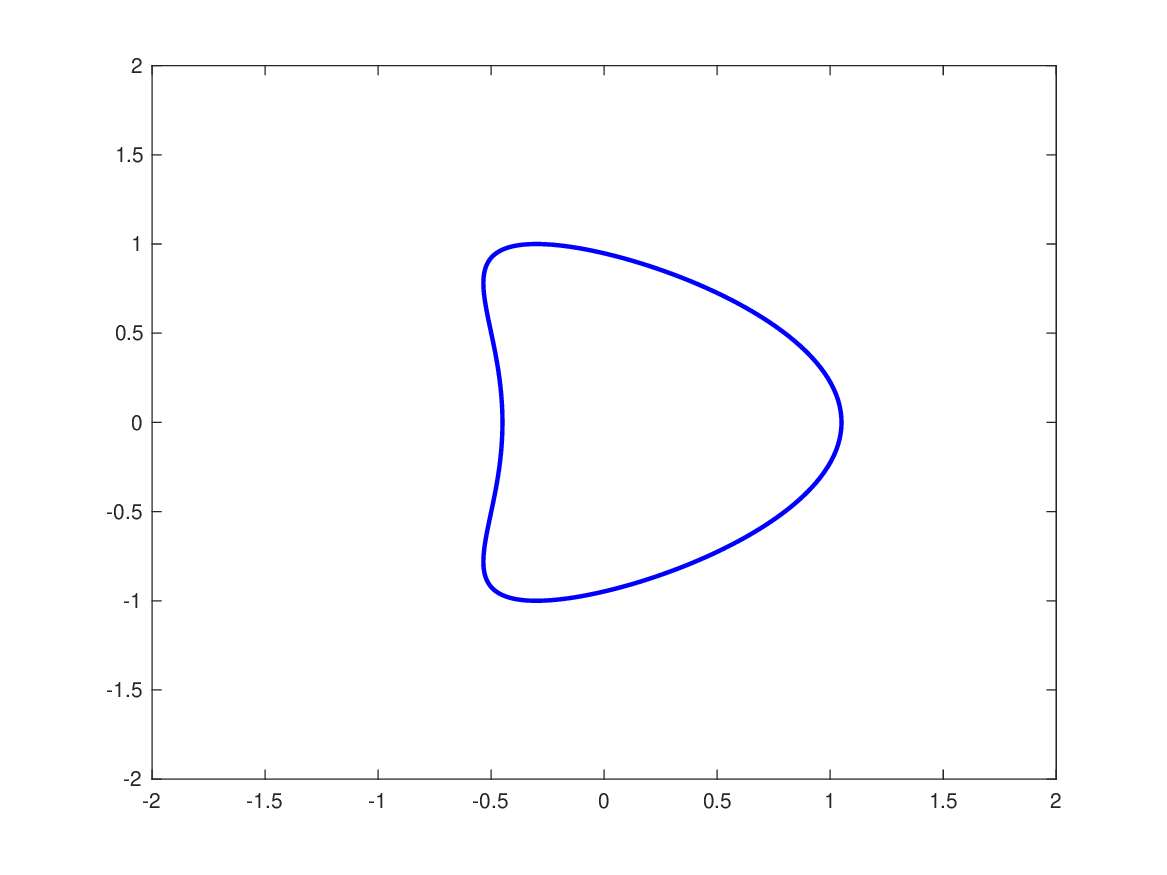}}
	\caption{Visual representations of the three scatterers defined in Table \ref{scatterers}.}\label{fig:scatterers}
\end{figure}

The numerical results for the three different scatterers is given in Table \ref{otherscatterers}. Here, we again present the first five eigenvalues for the aforementioned scatterers. This gives numerical evidence that the interlacing conjecture seems to hold for all three scatterers (and the unit disk). Another observation to make from Tables \ref{UnitDisk} and \ref{otherscatterers} is that the first clamped transmission eigenvalue is seen to be monotone with respect to meas$(D)$ as stated in the previous section. Indeed, we have that 
$$\text{meas}(\text{Kite}) \approx 2.356, \quad   \text{meas}(\text{Peanut}) \approx 1.963, \quad \text{ and } \quad \text{meas}(\text{Star}) \approx 0.758$$ 
and from our calculations we have that 
$$k_1 (\text{Kite})    \leq k_1(\text{Peanut})   \leq k_1(\text{Star}) . $$
For some simpler eigenvalue problems, interlacing inequalities and monotonicity with respect to meas$(D)$ can be obtained via the Rayleigh quotient. Due to the more complex structure of the clamped transmission eigenvalue problem, this approach is not as straightforward.

\begin{table}[h]
\centering
 \begin{tabular}{r|r|r|r|r|r|r}
\hline
\hline
                             & NE & 0.00000 & 3.51176 & 3.51176 & 5.32657 & 5.32657\\
Star--shaped  	    & TE  & 3.26716 & 6.18638 & 6.18638 & 8.70290 & 8.70290\\
		            & DE & 5.06979 & 8.00314 & 8.00314 & 10.45544 & 10.45544\\
\hline
                     		& NE & 0.00000 & 1.72126 & 3.02611 & 3.45854 & 3.66118\\
Peanut--shaped	& TE  & 2.13093 & 3.41900 & 4.70289 & 4.89266 & 5.55246\\
		    		& DE & 3.36058 & 4.38535 & 5.79406 & 6.08287 & 6.68797\\
\hline
                              &NE  & 0.00000 & 1.77091 & 2.18272 & 3.39190 & 3.52298\\
Kite--shaped          &TE  & 1.91665 & 3.38373 & 3.75151 & 4.96416 & 5.03581\\
                              &DE  & 2.95502 & 4.37204 & 4.78839 & 6.03903 & 6.66370\\
\hline
 \end{tabular}
 \caption{The first five Dirichlet, Neumann and Clamped Transmission Eigenvalues for the Peanut--shaped, Star--shaped and  Kite--shaped scatterers.} \label{otherscatterers}
\end{table}

\section{Conclusion}\label{conclusion}
In conclusion, we have provided an analytical and numerical study of the new clamped transmission eigenvalue problem. This eigenvalue problem is associated with scattering in a thin elastic plate. Even though the scatterer is bounded, the eigenvalue problem \eqref{tep1}--\eqref{tep2} is posed in the entirety of $\R^2$. With the results in this paper, we now know that there exists an infinite discrete set of real eigenvalues such that $k>0$. Here, we have provided more justification that the first clamped transmission eigenvalue is monotonically decreasing with respect to the measure of $D$, but it has not yet been proved. Another question that arises from Theorem \ref{upperbound} is if there are interlacing inequalities with other eigenvalues associated with the scatterer, which we have numerically investigated in Section \ref{numerics}. Also, the eigenvalue problem studied here is derived via a clamped obstacle assumption. It would be interesting to see what eigenvalue problems one derives from either the simply supported or free plate obstacle assumption. {\color{black}Lastly, spectral patterns of transmission eigenfunctions have received considerable attentions with striking applications in \cite{BS1,BS2,BS3,BS4} which can also be studied for this model.}\\

\noindent{\bf Acknowledgments:} The research of author I. Harris is partially supported by the NSF DMS Grants 2509722 and 2208256.


\end{document}